\documentclass[11pt,reqno]{amsart}
\usepackage{bbm}
\usepackage{amssymb, amsthm, amsmath, amscd}
\usepackage[usenames,dvipsnames]{color}

\theoremstyle{definition}
\newtheorem{notation}{Notation}[section]
\newtheorem{theorem}[notation]{Theorem}

\newtheorem{corollary}[notation]{Corollary}
\newtheorem{lemma}[notation]{Lemma}
\newtheorem{remark}[notation]{Remark}
\newtheorem{definition}[notation]{Definition}
\newtheorem{question}[notation]{Question}
\newtheorem*{theorem*}{Theorem}
\newtheorem*{corollary*}{Corollary}

\usepackage{amsfonts}
\usepackage{mathrsfs}
\usepackage{enumerate}
\usepackage{textcomp}
\usepackage[all]{xy}
\usepackage{enumitem}
\usepackage{mathtools}
\usepackage{hyperref}
\hypersetup{
	colorlinks=true,
	citecolor=Blue,
	linkcolor=Blue,
	pdfborder=0 0 0}

\numberwithin{equation}{section}

\begin{document}

\title{Stable rank for crossed products by finite group actions with the weak tracial Rokhlin property}

\author{Xiaochun Fang}
\address{School of Mathematical Sciences, Tongji University, Shanghai 200092, China}
\email{xfang@tongji.edu.cn}

\author{Zhongli Wang}
\address{School of Mathematical Sciences, Tongji University, Shanghai 200092, China}
\email{1810411@tongji.edu.cn}

\date{Month, Day, Year}
\date{\today}

\maketitle
\begin{abstract}
Let $A$ be an infinite-dimensional stably finite simple unital C*-algebra, let $G$ be a finite group, and let $\alpha\colon G\rightarrow \mathrm{Aut}(A)$ be an action of $G$ on $A$ which has the weak tracial Rokhlin property. We prove that if $A$ has property (TM), then the crossed product $A\rtimes_\alpha G$ has property (TM). As a corollary, if $A$ is an infinite-dimensional separable simple unital C*-algebra which has stable rank one and strict comparison, $\alpha\colon G\rightarrow \mathrm{Aut}(A)$ is an action of a finite group $G$ on $A$ with the weak tracial Rokhlin property, then $A\rtimes_\alpha G$ has stable rank one.
\end{abstract}








\section{Introduction}\label{sec1}
The notion of stable rank for C*-algebras was introduced by Rieffel in \cite{MR0693043}. It can be viewed as a noncommutative analogue of topological dimension for C*-algebras. The class of C*-algebras with minimal stable rank, i.e., stable rank one is of particular interest. It is shown in \cite{MR2106263} (in the unital case) and \cite{MR4524191} (in the non-unital case) that every finite simple $\mathcal{Z}$-stable C*-algebra has stable rank one.

The study of group actions on C*-algebras and the structure of crossed products is an important line of research. It has been a long-standing open problem in this area whether stable rank one is preserved under taking crossed products by finite group actions. Specifically, Phillips asked in \cite[Problem 10.2.6]{MR3837150} that if $A$ is a simple unital C*-algebra with stable rank one and $\alpha$ is an action of a finite group $G$ on $A$, does it follow that $A\rtimes_\alpha G$ has stable rank one? Without assuming simplicity for $A$, this is known to be false; see \cite[Example 8.2.1]{MR1053492}.

In order to obtain good results on the structure of crossed products, one generally requires that the group actions satisfy certain regularity conditions. The most useful conditions are the Rokhlin property, the tracial Rokhlin property, the weak tracial Rokhlin property and the finite Rokhlin dimension. It was proved by Osaka and Phillips in \cite{MR2875821} that various structural properties pass from a unital C*-algebra to its crossed product by a finite group action with the Rokhlin property, including stable rank one. However, the Rokhlin property is very restrictive since it implies certain divisibility conditions on K-theory. Inspired by the notion of tracial rank zero, Phillips introduced the tracial Rokhlin property for finite group actions in \cite{MR2808327}. The tracial Rokhlin property is much more common and still useful in the study of preservation of certain properties. For instance, it was proved by Fan and Fang in \cite[Theorem 3.1]{MR2487595} that the crossed product of an infinite-dimensional separable simple unital C*-algebra with stable rank one by a finite group action with the tracial Rokhlin property again has stable rank one. The same result was obtained by Osaka and Teruya in \cite[Theorem 3.4]{MR3759005}.

The weak tracial Rokhlin property (see Definition \ref{defwtr} below) is a generalization of the tracial Rokhlin property in which projections are replaced by positive elements. Since the tracial Rokhlin property requires the existence of projections, many C*-algebras do not admit any action with the tracial Rokhlin property, but they may admit an action with the weak tracial Rokhlin property; see \cite[Example 5.10]{MR3063095}. There are some classes of simple unital C*-algebras which are preserved under taking crossed products by finite group actions with the weak tracial Rokhlin property. For instance, the class of simple unital C*-algebras with tracial rank zero (see \cite[Theorem 2.6]{MR2808327} and \cite[Theorem 1.9]{MR3347172}), the class of simple unital tracially $\mathcal{Z}$-absorbing C*-algebras (see \cite[Theorem 5.6]{MR3063095}), and the class of simple unital C*-algebras with strict comparison (see \cite[Theorem 4.5]{MR4336489}). However, for the class of simple unital C*-algebras with stable rank one, the following question posed in \cite{MR4336489} is still open. 

\begin{question}[{\cite[Question 7.4]{MR4336489}}]\label{qs}
Let $A$ be an infinite-dimensional simple unital C*-algebra with stable rank one, let $G$ be a finite group, and $\alpha\colon G\rightarrow \mathrm{Aut}(A)$ be an action with the weak tracial Rokhlin property. Does it follow that $A\rtimes_\alpha G$ and $A^\alpha$ have stable rank one?
\end{question}

In this paper, we consider the notion of property (TM), a property of tracial matricial structure, introduced in \cite{fu2024tracialoscillationzerostable}. It is shown in \cite{fu2024tracialoscillationzerostable} that if $A$ is a separable simple C*-algebra which admits at least one densely defined non-trivial 2-quasitrace and has strict comparison, then $A$ has stable rank one if and only if $A$ has property (TM). Thus we deal with the preservation of property (TM) for finite group actions with the weak tracial Roklin property.

\begin{theorem*}[\ref{thmtm}]
	Let $A$ be an infinite-dimensional stably finite simple unital C*-algebra, let $G$ be a finite group, and let $\alpha\colon G\rightarrow \mathrm{Aut}(A)$ be an action of $G$ on $A$ which has the weak tracial Rokhlin property. If $A$ has property (TM), then $A\rtimes_\alpha G$ has property (TM).
\end{theorem*}

Combining this theorem and the results of \cite{fu2024tracialoscillationzerostable}, we give an affirmative answer to Question \ref{qs} under the assumption that $A$ is separable and has strict comparison.

\begin{corollary*}[\ref{corsr}]
	Let $A$ be an infinite-dimensional separable simple unital C*-algebra, let $G$ be a finite group, and let $\alpha\colon G\rightarrow \mathrm{Aut}(A)$ be an action of $G$ on $A$ which has the weak tracial Rokhlin property. If $A$ has stable rank one and strict comparison, then $A\rtimes_\alpha G$ and $A^\alpha$ have stable rank one.
\end{corollary*}

This paper is organized as follows. Section \ref{sec2} contains information on Cuntz comparison, order zero maps, several approximation lemmas, quasitraces, and property (TM). In Section \ref{sec3},  we present an approximation property for the crossed products by finite group actions with the weak tracial Roklin property. Section \ref{sec4} contains the proofs of the main theorem and corollary.

\section{Preliminaries}\label{sec2}
\subsection{Some basic notations}
We use the following notations in this paper.

\begin{notation} 
	Let $A$ be a C*-algebra. Denote by $A_+$ the set of all positive elements in $A$, and by $A^1$ the closed unit ball of $A$. Put $A_+^1=A_+\cap A^1$. Let $S\subset A$ be a subset. Let $a,b\in A$ and $\varepsilon>0$. We write $a\approx_\varepsilon b$ if $\|a-b\|<\varepsilon$. We write $a\in_\varepsilon S$ if there is an element $c\in S$ such that $\|a-c\|<\varepsilon$. Denote by $\mathrm{Ped}(A)$ the Pedersen ideal, i.e., the minimal dense ideal of $A$ (see \cite[Section 5.6]{MR0548006}). Put $\mathrm{Ped}(A)_+=\mathrm{Ped}(A)\cap A_+$, $\mathrm{Ped}(A)^1=\mathrm{Ped}(A)\cap A^1$ and $\mathrm{Ped}(A)_+^1=\mathrm{Ped}(A)\cap A_+^1$.
\end{notation}	
	
\begin{notation}
	Let $A$ be a C*-algebra, let $a\in A_+$, and let $\varepsilon>0$. Let $f\colon [0,\infty)\rightarrow [0,\infty)$ be the function $f(t)=\max\{0,t-\varepsilon\}$. Define $(a-\varepsilon)_+=f(a)$. 
\end{notation}

\begin{notation}
	Let $\varepsilon>0$. Define a continuous function $f_\varepsilon\colon [0,\infty)\rightarrow [0,1]$ by
	\[
	f_\varepsilon(t)=
	\begin{cases}
		0 & 0\leq t< \varepsilon/2,\\
		\rm{linear} & \varepsilon/2\leq t <\varepsilon,\\
		1 & t\geq \varepsilon.
	\end{cases}
	\]
\end{notation}

\begin{notation}
	Let $A$ be a C*-algebra. Denote by $K$ the C*-algebra of compact operators on an infinite-dimensional separable Hilbert space. Let $A\otimes K$ be the direct limit of the system $(M_n(A))_{n=1}^\infty$ using the usual upper-left corner embeddings. We will make the identifications $A\subset M_n(A)\subset A\otimes K$ whenever it is convenient. We write $\{e_{i,j}\}_{i,j=1}^n$ for the standard system of matrix units for $M_n$ and $1_n$ for the identity of $M_n$. We also write $\{e_{i,j}\}$ for the standard system of matrix units for $K$.
	
	Let $a\in M_n(A)$ and $b\in M_m(A)$. Denote by $a\oplus b$ the element
	\[
	\begin{pmatrix}
		a & 0 \\
		0 & b 
	\end{pmatrix}
	\in M_{n+m}(A).
	\]
\end{notation}

\subsection{Cuntz subequivalence}
The following definition is originally from \cite{MR0467332}.

\begin{definition}
	Let $A$ be a C*-algebra, and let $a,b\in (A\otimes K)_+$. The element $a$ is said to be Cuntz subequivalent to $b$ in $A$, written $a\precsim_A b$, if there is a sequence $(x_n)_{n=1}^\infty$ in $A\otimes K$ such that
	\[
	\lim_{n\rightarrow \infty}\|a-x_n^*bx_n\|=0.
	\]
	The element $a$ is said to be Cuntz equivalent to $b$ in $A$, written $a\sim_A b$, if both $a\precsim_A b$ and $b\precsim_A a$. When there is no confusion about the algebra $A$, we suppress it in the notation.
\end{definition}

\begin{remark}[{see \cite[Remark 1.2]{1408.5546}}]
	If $a,b\in M_n(A)_+$ for some $n\in \mathbb{N}$ and $a\precsim_A b$, then there is a sequence $(x_n)_{n=1}^\infty$ in $M_n(A)$ such that
	\[
	\lim_{n\rightarrow \infty}\|a-x_n^*bx_n\|=0.
	\]
\end{remark}

The following lemma contains some of the known results about Cuntz subequivalence which we will use. Part (1) is \cite[Lemma 2.2(i)]{MR1759891}, Part (2) is \cite[Lemma 2.5(ii)]{MR1759891}, Part (3) is \cite[Proposition 2.7(i)]{MR1759891}, Part (4) is \cite[Lemma2.8(iii)]{MR1759891}, and Part (5) is \cite[Lemma2.9]{MR1759891}.

\begin{lemma}\label{lemcuc} Let $A$ be a C*-algebra.
	\begin{enumerate}
		\item Let $a\in A_+$ and let $f\colon[0,\infty)\rightarrow [0,\infty)$ be a continuous function with $f(0)=0$. Then $f(a)\precsim a$.
		\item Let $a,b\in A_+$ and let $\varepsilon>0$. If $\|a-b\|<\varepsilon$, then $(a-\varepsilon)_+\precsim b$.
		\item Let $a,b\in A_+$. If $a\in \overline{bAb}$, then $a\precsim b$.
		\item Let $a,b\in A_+$. If $ab=0$, then $a+b\sim a\oplus b$.
		\item Let $a_1,a_2,b_1,b_2\in A_+$. If $a_1\precsim a_2$ and $b_1\precsim b_2$, then $a_1\oplus b_1\precsim a_2\oplus b_2$.
	\end{enumerate}
\end{lemma}

The following lemma is taken from \cite[Lemma 3.3]{MR4107516}. Note that for any $a\in A_+$ and any $\sigma>0$, the hereditary C*-subalgebra of $A$ generated by $(a-\sigma)_+$ and the hereditary C*-subalgebra of $A$ generated by $f_{2\sigma}(a)$ are the same.

\begin{lemma}[\cite{MR4107516}]\label{lemcup}
	Let $\varepsilon>0$ and $\sigma>0$ be given. There exists $\delta>0$ satisfying the following condition: Whenever $A$ is a C*-algebra and $a,b\in A^1_+$ satisfy
	\[
	\|a-b\|<\delta,
	\]
	then there exists an injective $\ast$-homomorphism $\varphi\colon \overline{(a-\sigma)_+A(a-\sigma)_+}\rightarrow \overline{bAb}$ satisfying
	\[
	\|\varphi(x)-x\|<\varepsilon\|x\|
	\]
	for all $ x\in \overline{(a-\sigma)_+A(a-\sigma)_+}$.
\end{lemma}

\subsection{Order zero maps}

\begin{definition}[{\cite[Definition 2.3]{MR2545617}}]
	Let $A$ and $B$ be C*-algebras and let $\varphi\colon A\rightarrow B$ be a completely positive, abbreviated to c.p., map (c.p.c. if the map is contractive). The map $\varphi$ is said to have order zero, if for all $a,b\in A_+$,
	\[
	ab=0 \Rightarrow \varphi(a)\varphi(b)=0.
	\]
\end{definition}

The following is the structure theorem for c.p. order zero maps.

\begin{theorem}[{\cite[Theorem 3.3]{MR2545617}}]\label{thmozs}
	Let $A$ and $B$ be C*-algebras, and let $\varphi\colon A\rightarrow B$ be a c.p. order zero map. Set $C=C^*(\varphi(A))\subset B$. Then there is a positive element $h\in \mathcal{M}(C)\cap C'$ with $\|h\|=\|\varphi\|$ and a $\ast$-homomorphism $\pi\colon A\rightarrow \mathcal{M}(C)\cap \{h\}'$, where $\mathcal{M}(C)$ is the multiplier algebra of $C$, such that
	\[
	\varphi(a)=h\pi(a)
	\]
	for all $a\in A$. If $A$ is unital, then $h=\varphi(1_A)\in C$.
\end{theorem}

The following is known as the functional calculus for order zero maps.

\begin{definition}[{\cite[Corollary 4.2]{MR2545617}}]\label{cpf}
	Let $\varphi\colon A\rightarrow B$ be a c.p.c. order zero map, let $C, h, \pi$ be as above, and let $f\in C_0(0,1]$ be a positive function of norm at most one. Then the map $f(\varphi)\colon A\rightarrow B$ defined by
	\[
	f(\varphi)(a)=f(h)\pi(a)
	\]
	for all $a\in A$ is a well-defined c.p.c. order zero map.
\end{definition}

The following lemma is a simple observation.
\begin{lemma}\label{lemcpf}
	Let $\varphi\colon A\rightarrow B$ be a c.p.c. order zero map, and let $f\in C_0(0,1]$ be a positive function of norm at most one. If $p\in A$ is a projection, then
	\[
	f(\varphi)(p)=f(\varphi(p)).
	\]
\end{lemma}

\begin{proof}
	Let $h, \pi$ be as in Theorem \ref{thmozs}. If $f(t)=t^n$ for some $n\in \mathbb{N}$, then
	\[
	f(\varphi(p))=(h\pi(p))^n=h^n\pi(p)=f(\varphi)(p).
	\]
	
	Therefore, the statement holds for any polynomial with zero constant term. In general, there exists a sequence of polynomials with zero constant term which converges uniformly to $f$ on $[0,1]$. The statement follows from the spectral theorem.
\end{proof}

\subsection{Approximation lemmas}
This subsection contains several known approximation results which will be needed. Part (1) and (3) of the following are \cite[Lemma 1.5 and Lemma 1.6]{MR4078704}. Part (2) is a special case of \cite[Lemma 2.5.11(2)]{MR1884366} but without proof.

\begin{lemma}\label{approx}
	\
	\begin{enumerate}
		\item Let $f\colon [0,1]\rightarrow \mathbb{C}$ be a continuous function. Then for any $\varepsilon>0$, there exists $\delta>0$ such that whenever $A$ is a C*-algebra and $a\in A_+^1$, $b\in A^1$ satisfy $\|ab-ba\|<\delta$, then
		\[
		\|f(a)b-bf(a)\|<\varepsilon.
		\]
		\item Let $f\colon [-1,1]\rightarrow \mathbb{C}$ be a continuous function. Then for any $\varepsilon>0$, there exists $\delta>0$ such that whenever $A$ is a C*-algebra and $a,b\in A^1$ are self-adjoint elements satisfying $\|a-b\|<\delta$, then
		\[
		\|f(a)-f(b)\|<\varepsilon.
		\]
		\item Let $f\colon [0,1]\rightarrow \mathbb{C}$ be a continuous function with $f(0)=0$. Then for any $\varepsilon>0$, there exists $\delta>0$ such that whenever $A$ is a C*-algebra, $B\subset A$ is a subalgebra, and $a\in A^1$ and $b\in B_+^1$ satisfy $ba\in_\delta B$, then
		\[
		f(b)a\in_\varepsilon B.
		\]
	\end{enumerate}
\end{lemma}

\begin{proof}
	(2) Choose $n\in \mathbb{N}$ and a polynomial $p(t)=\sum_{i=0}^n\lambda_it^i$ with coefficients $\lambda_i\in\mathbb{C}$ for all $i=0,1,\cdots,n$ such that
	\[
	|p(t)-f(t)|<\varepsilon/3
	\]
	for all $t\in [-1,1]$. Define
	\[
	M=\max\{|\lambda_0|,|\lambda_1|,\cdots,|\lambda_n|\}+1 \text{ and } \delta=\frac{\varepsilon}{3Mn^2}.
	\]
	
	Let $A,a$ and $b$ be as in the hypotheses. For any $0\leq i\leq n$, we have
	\begin{align*}
		\|a^i-b^i\| & \leq \|a^i-a^{i-1}b\|+\|a^{i-1}b-a^{i-2}b^2\|+\cdots+\|ab^{i-1}-b^i\|\\
		& \leq\|a^{i-1}\|\|a-b\|+\|a^{i-2}\|\|a-b\|\|b\|+\cdots+\|a-b\|\|b^{i-1}\|\\
		& \leq i\|a-b\|<n\delta.
	\end{align*}
	Therefore,
	\[
	\|p(a)-p(b)\|\leq \sum_{i=0}^n|\lambda_i|\|a^i-b^i\|<Mn^2\delta=\varepsilon/3,
	\]
	and
	\begin{align*}
		\|f(a)-f(b)\| & \leq \|f(a)-p(a)\|+\|p(a)-p(b)\|+\|p(b)-f(b)\|\\
		& < \varepsilon/3+\varepsilon/3+\varepsilon/3=\varepsilon.
	\end{align*}
\end{proof}

\subsection{Quasitraces}
The following definition is from \cite[Definition 2.22]{MR2032998}.

\begin{definition}
	Let $A$ be a C*-algebra. A quasitrace on $A$ is a function $\tau\colon A_+\rightarrow [0,\infty]$ with $\tau(0)=0$ such that
	\begin{enumerate}
		\item $\tau(x^*x)=\tau(xx^*)$ for all $x\in A$.
		\item $\tau(a+b)=\tau(a)+\tau(b)$ for all commuting elements $a,b\in A_+$.
	\end{enumerate}
	A 2-quasitrace on $A$ is a quasitrace $\tau$ that extends to a quasitrace $\tau_2$ on $M_2(A)$ with $\tau_2(a\otimes e_{1,1})=\tau(a)$ for all $a\in A_+$. Set $\mathrm{Dom}_{1/2}(\tau)=\{a\in A\colon \tau(a^*a)<\infty\}$. A quasitrace $\tau$ is said to be densely defined if $\mathrm{Dom}_{1/2}(\tau)$ is dense in $A$, and $\tau$ is said to be bounded if $\mathrm{Dom}_{1/2}(\tau)=A$.
\end{definition}

Denote by $\widetilde{\mathrm{QT}}(A)$ the set of densely defined, lower semicontinuous 2-quasitrace on $A$.

\begin{remark}[{see \cite[Remark 2.27]{MR2032998}}]\label{remqt}
    (1) If $\tau$ is a lower semicontinuous 2-quasitrace on $A$, then $\tau$ is order-preserving on $A_+$.
    
    (2) Every lower semicontinuous 2-quasitrace $\tau$ on $A$ has a unique extension to a lower semicontinuous quasitrace $\tau_\infty$ on $A\otimes K$ with $\tau_\infty(a\otimes e_{1,1})=\tau(a)$ for all $a\in A_+$. In what follows, we will identify $\tau$ with $\tau_\infty$.
    
    (3) Let $A$ be a simple C*-algebra and let $a\in A_+\setminus\{0\}$. Then every $\tau\in \widetilde{\mathrm{QT}}(\overline{aAa})$ can be uniquely extended to a $\tilde{\tau}\in \widetilde{\mathrm{QT}}(A)$. Clearly, for every $\tilde{\tau}\in \widetilde{\mathrm{QT}}(A)$, the restriction $\tilde{\tau}\vert_{\overline{aAa}}\in\widetilde{\mathrm{QT}}(\overline{aAa})$. In what follows, we will identify $\widetilde{\mathrm{QT}}(\overline{aAa})$ with $\widetilde{\mathrm{QT}}(A)$.
\end{remark}

\begin{definition}
	For any $\tau \in \widetilde{\mathrm{QT}}(A)$ and any C*-subalgebra $B\subset A\otimes K$, define
	\[
	\|\tau\vert_B\|=\sup\{\tau(a)\colon a\in B_+^1\}.
	\]
\end{definition}

Denote by $\mathrm{QT}(A)$ the set of bounded 2-quasitrace on $A$ with $\|\tau\vert_A\|=1$.

\begin{theorem}[\cite{MR0616272}, \cite{MR0650185}]\label{thmsfqt}
	Let $A$ be a stably finite unital C*-algebra. Then $\mathrm{QT}(A)\neq \emptyset$.
\end{theorem}

\begin{definition}
	Let $A$ be a C*-algebra with $\widetilde{\mathrm{QT}}(A)\neq\emptyset$ and let $S\subset \widetilde{\mathrm{QT}}(A)$ be a non-empty subset. For any $a\in A\otimes K$, define
	\[
	\|a\|_{2,S}=\sup\{(\tau(a^*a))^{1/2}\colon \tau\in S\}.
	\]
\end{definition}

\begin{lemma}[{cf. \cite[Lemma 3.5]{MR3241179}}]\label{lemqtn}
	Let $A$ be a C*-algebra and let $\tau\in \widetilde{\mathrm{QT}}(A)$. Then
	\begin{enumerate}
		\item $\tau(a+b)^{1/2}\leq \tau(a)^{1/2}+\tau(b)^{1/2}$ for all $a,b\in (A\otimes K)_+$.
		\item $\|x+y\|_{2, \tau}^{2/3}\leq \|x\|_{2, \tau}^{2/3}+\|y\|_{2, \tau}^{2/3}$ for all $x,y\in A\otimes K$.
		\item $\|xy\|_{2, \tau}\leq \|x\|\|y\|_{2, \tau}$ and $\|xy\|_{2, \tau}\leq \|x\|_{2, \tau}\|y\|$ for all $x,y\in A\otimes K$.
	\end{enumerate}
\end{lemma}

\begin{proof}
	(1) Note that if $a,b\in M_m(A)_+$ for some $m\in \mathbb{N}$, then the proof of \cite[Lemma 3.5(1)]{MR3241179} still works in the unbounded case. In general, let $a,b\in (A\otimes K)_+$, and let $\{e_{i,j}\}$ be the standard matrix units for $K$. Set $p_n=\sum_{i=1}^n1_{\mathcal{M}(A)}\otimes e_{i,i}\in \mathcal{M}(A\otimes K)$ for each $n\in \mathbb{N}$. Then $(p_n)_{n=1}^\infty$ is an increasing sequence of projections in $\mathcal{M}(A\otimes K)$ which converges to $1_{\mathcal{M}(A\otimes K)}$ in the strict topology of $\mathcal{M}(A\otimes K)$. Note that $p_nap_n,p_nbp_n\in M_n(A)_+$. So we have
	\[
	\begin{aligned}
		\tau((a+b)^{1/2}p_n(a+b)^{1/2})^{1/2} & =\tau(p_n(a+b)p_n)^{1/2}\\
		& =\tau(p_nap_n+p_nbp_n)^{1/2}\\
		& \leq \tau(p_nap_n)^{1/2}+\tau(p_nbp_n)^{1/2}\\
		& =\tau(a^{1/2}p_na^{1/2})^{1/2}+\tau(b^{1/2}p_nb^{1/2})^{1/2}\\
		& \leq \tau(a)^{1/2}+\tau(b)^{1/2}.
	\end{aligned}
	\]
	Since $\tau$ is lower semicontinuous, we get
	\[
	\tau(a+b)^{1/2}=\lim_{n\rightarrow \infty}\tau((a+b)^{1/2}p_n(a+b)^{1/2})^{1/2}\leq \tau(a)^{1/2}+\tau(b)^{1/2}.
	\]
	
	Since $\tau$ is order-preserving on $(A\otimes K)_+$, the proof of (2) and (3) is the same as that of \cite[Lemma 3.5]{MR3241179}.
\end{proof}

\begin{definition}
	Let $A$ be a C*-algebra and let $\tau\in \widetilde{\mathrm{QT}}(A)$. The induced dimension function $d_\tau\colon (A\otimes K)_+\rightarrow [0,\infty]$ is given by
	\[
	d_\tau(a)=\lim_{n\rightarrow \infty}\tau(a^{1/n}).
	\]
\end{definition}

The following result is well known.

\begin{lemma}\label{lemqtnd}
	Let $A$ be a C*-algebra, let $a\in (A\otimes K)_+$, and let $\tau\in \widetilde{\mathrm{QT}}(A)$. Then
	\[
	\|\tau\vert_{\overline{a(A\otimes K)a}}\|=d_\tau(a).
	\]
\end{lemma}
\begin{proof}
	Since $d_\tau(\lambda a)=d_\tau(a)$ for all $\lambda>0$, we may assume that $\|a\|\leq 1$. Note that $(a^{1/n})_{n=1}^\infty$ is an increasing approximate identity for $\overline{a(A\otimes K)a}$. For any $b\in (\overline{a(A\otimes K)a})_+^1$, using the lower semicontinuity of $\tau$, we get
	\[
	\tau(b)=\lim_{n\rightarrow \infty}\tau(b^{1/2}a^{2/n}b^{1/2})=\lim_{n\rightarrow \infty}\tau(a^{1/n}ba^{1/n})\leq \lim_{n\rightarrow \infty}\tau(a^{2/n})=d_\tau(a).
	\]
	Thus $\|\tau\vert_{\overline{a(A\otimes K)a}}\|\leq d_\tau(a)$. On the other hand, we have $\tau(a^{1/n})\leq \|\tau\vert_{\overline{a(A\otimes K)a}}\|$ for all $n\in \mathbb{N}$. Hence $\|\tau\vert_{\overline{a(A\otimes K)a}}\|\geq d_\tau(a)$.
\end{proof}

\subsection{Property (TM)}
The following definition is from \cite[Definition 8.1]{fu2024tracialoscillationzerostable}.

\begin{definition}\label{deftm}
	Let $A$ be a simple unital C*-algebra with $\mathrm{QT}(A)\neq \emptyset$. Then $A$ is said to have property (TM), if for any $a\in \mathrm{Ped}(A\otimes K)^1_+$, any $n\in \mathbb{N}$, and any $\varepsilon>0$, there exists a c.p.c. order zero map $\varphi\colon M_n\rightarrow \overline{a(A\otimes K)a}$ such that
	\[
	\|a-\varphi(1_n)a\|_{2,\mathrm{QT}(A)}<\varepsilon.
	\]
\end{definition}

\begin{remark}
	In \cite{fu2024tracialoscillationzerostable}, the original definition of property (TM) is defined for $\sigma$-unital simple C*-algebras as follows:
	
	A $\sigma$-unital simple C*-algebra $A$ with $\widetilde{\mathrm{QT}}(A)\setminus\{0\}\neq\emptyset$ is said to have property (TM), if there exists $e\in \mathrm{Ped}(A)_+^1\setminus\{0\}$ such that the following condition holds: For any $a\in \mathrm{Ped}(\overline{eAe}\otimes K)^1_+$, any $n\in \mathbb{N}$, and any $\varepsilon>0$, there exists a c.p.c. order zero map $\varphi\colon M_n\rightarrow \overline{a(\overline{eAe} \otimes K)a}$ such that
	\[
	\|a-\varphi(1_n)a\|_{2,\overline{\mathrm{QT}(\overline{eAe})}^w}<\varepsilon,
	\]
	where $\overline{\mathrm{QT}(\overline{eAe})}^w$ is the closure of $\mathrm{QT}(\overline{eAe})$ in $\widetilde{\mathrm{QT}}(A)$ under the topology of pointwise convergence on $\mathrm{Ped}(A\otimes K)_+$.
	
	First, note that if $A$ and $e$ satisfy the condition in the above definition, then any $e_1\in \mathrm{Ped}(A)_+^1\setminus\{0\}$ would satisfy the same condition. Indeed, given another $e_1\in \mathrm{Ped}(A)_+^1\setminus\{0\}$, by \cite[Proposition 2.10(4)]{fu2024tracialoscillationzerostable} there exists $L$ depending on $e$ and $e_1$ such that
	\[
	\|x\|_{2,\overline{\mathrm{QT}(\overline{e_1Ae_1})}^w}\leq L\|x\|_{2,\overline{\mathrm{QT}(\overline{eAe})}^w}
	\]
	for all $x\in \mathrm{Ped}(A\otimes K)$. By Brown's theorem (\cite[Theorem 2.8]{MR0454645}), there is an isomorphism $\psi\colon \overline{eAe}\otimes K\rightarrow \overline{e_1Ae_1}\otimes K$. This isomorphism is induced by a partial isometry $v$ in $\mathcal{M}(A\otimes K)$ which takes $x$ to $vxv^*$. So it is easily checked that
	\[
	\tau(\psi(x))=\tau(x)
	\]
	for all $\tau\in \widetilde{\mathrm{QT}}(A)$ and all $x\in (\overline{eAe}\otimes K)_+$.
	
	Since $\psi$ is surjective, $\psi(\mathrm{Ped}(\overline{eAe}\otimes K))=\mathrm{Ped}(\overline{e_1Ae_1}\otimes K)$ (see \cite[II.5.2.7]{MR2188261}). Given any $\psi(a)\in \mathrm{Ped}(\overline{e_1Ae_1}\otimes K)^1_+$, any $n\in \mathbb{N}$ and any $\varepsilon>0$, applying the hypothesis with $\varepsilon/L$ in place of $\varepsilon$ and everything else as given, we get a c.p.c. order zero map $\varphi\colon M_n\rightarrow \overline{a(\overline{eAe} \otimes K)a}$ such that
	\[\|a-\varphi(1_n)a\|_{2,\overline{\mathrm{QT}(\overline{eAe})}^w}<\varepsilon/L.
	\]
	Considering the c.p.c. order zero map $\psi\circ\varphi\colon M_n\rightarrow \overline{\psi(a)(\overline{e_1Ae_1} \otimes K)\psi(a)}$, we have
	\begin{align*}
		\|\psi(a)-\psi\circ\varphi(1_n)\psi(a)\|_{2,\overline{\mathrm{QT}(\overline{e_1Ae_1})}^w}
		& =\|a-\varphi(1_n)a\|_{2,\overline{\mathrm{QT}(\overline{e_1Ae_1})}^w}\\
		& \leq L\|a-\varphi(1_n)a\|_{2,\overline{\mathrm{QT}(\overline{eAe})}^w}<\varepsilon.
	\end{align*}
	This shows that the above definition is independent of the choice of $e$.
	
	If $A$ is unital, then $\mathrm{QT}(A)$ is a compact subset of $\widetilde{\mathrm{QT}}(A)$ (see \cite[Theorem 4.4]{MR2823868}). By choosing $e=1_A$, we get Definition \ref{deftm}.
\end{remark}

\section{Approximation property for the crossed product}\label{sec3}

In this section, based on the work of \cite{MR2712082}, we present an approximation property for the crossed products by finite group actions with the weak tracial Roklin property (see Theorem \ref{thmap}). This approximation property is closely related to the notion of essential tracial approximation (see \cite[Definition 3.1]{MR4387777}) and the notion of generalized tracial approximation (see \cite[Definition 1.2]{elliott2023nonunitalgeneralizedtracially}).

\begin{definition}[{\cite[Definition 3.2]{MR4336489}}]\label{defwtr}
	Let $A$ be an infinite-dimensional simple unital C*-algebra, let $G$ be a finite group, and let $\alpha\colon G\rightarrow \mathrm{Aut}(A)$ be an action of $G$ on $A$. Then $\alpha$ is said to have the weak tracial Rokhlin property if for any finite subset $F\subset A$, any $\varepsilon>0$, and any $s\in A_+$ with $\|s\|=1$, there exist orthogonal positive elements $f_g\in A_+^1$ for $g\in G$ such that
	\begin{enumerate}
		\item $\|f_gx-xf_g\|<\varepsilon$ for all $x\in F$ and all $g\in G$.
		\item $\|\alpha_g(f_h)-f_{gh}\|<\varepsilon$ for all $g,h\in G$.
		\item $1-f\precsim_A s$, where $f=\sum_{g\in G}f_g$.
		\item $\|fsf\|>1-\varepsilon$.
	\end{enumerate}
\end{definition}

\begin{remark}\label{remsim}
	Let $A, G, \alpha$ be as above. If $\alpha$ has the weak tracial Rokhlin property, then $\alpha_g$ is outer for all $g\in G\setminus\{1\}$ (see \cite[Proposition 3.2]{MR4216445}). It follows from \cite[Theorem 3.1]{MR0634163} that $A\rtimes_\alpha G$ is simple.
\end{remark}

The following result is a modified version of \cite[Lemma VII.4]{MR2712082}. Note that Condition (1) and (2) of Lemma \ref{lemap} have been obtained in \cite[Lemma VII.4]{MR2712082}. Since the proof is similar to that of \cite[Lemma VII.4]{MR2712082}, we omit some details.

\begin{lemma}[{cf. \cite[Lemma VII.4]{MR2712082}}]\label{lemap}
	Let $A$ be an infinite-dimensional simple unital C*-algebra, let $G$ be a finite group with $\mathrm{card}(G)=n$, and let $\alpha\colon G\rightarrow \mathrm{Aut}(A)$ be an action of $G$ on $A$ which has the weak tracial Rokhlin property. Then for any finite subet $F\subset A\rtimes_\alpha G$, any $\varepsilon>0$, and any $s\in (A\rtimes_\alpha G)_+\setminus \{0\}$, there exist an element $a\in A_+^1$, a C*-subalgebra $B$ of $A\rtimes_\alpha G$ with $B\cong \overline{aAa}\otimes M_n$, and an element $d\in B_+^1$, such that
\begin{enumerate}
	\item\label{lem3-1} $\|dx-xd\|<\varepsilon$ for all $x\in F$.
	\item\label{lem3-2} $dx\in_\varepsilon B$ for all $x\in F$.
	\item\label{lem3-3} $1-d\precsim_{A\rtimes_\alpha G} s$.
	\item\label{lem3-4} $\|d\|=1$.
\end{enumerate}
\end{lemma}

\begin{proof}
	Let $F\subset A\rtimes_\alpha G $ be a finite subset, let $0<\varepsilon<1$, and let $s\in (A\rtimes_\alpha G)_+\setminus \{0\}$. For $g\in G$, let $u_g\in A\rtimes_\alpha G$ be the canonical unitary implementing the automorphism $\alpha_g$. Without loss of generality, we may assume that $F=\{xu_g\colon x\in F_1, g\in G\}$, where $F_1$ is a finite subset of $A^1$.
	
	Since $\alpha$ has the weak tracial Rokhlin property, $\alpha_g$ is outer for all $g\in G\setminus \{1\}$. By \cite[Lemma 5.1]{MR3063095}, there exists an element $s_1\in A_+\setminus \{0\}$ such that $s_1\precsim_{A\rtimes_\alpha G} s$. We may assumed that $\|s_1\|=1$.
	
	By Lemma \ref{approx}(3), we can choose $\delta_1$ with $0<\delta_1<\varepsilon$ such that whenever $D\subset A\rtimes_\alpha G$ is a subalgebra, and $y\in (A\rtimes_\alpha G)^1$ and $z\in D_+^1$ satisfy $zy\in_{\delta_1} D$, then $z^{1/3}y\in_{\varepsilon/3} D$.
	
	Choose $\delta_2$ such that $0<\delta_2<\delta_1/(n^2+17n+5)$. Define continuous funcitons $l_1,l_2\colon [0,1]\rightarrow [0,1]$ by
	\[
	l_1(t)=
	\begin{cases}
		(1-3\delta_2)^{-1}t & 0\leq t< 1-3\delta_2,\\
		1 & 1-3\delta_2\leq t \leq1,
	\end{cases}
	\]
	and
	\[
	l_2(t)=
	\begin{cases}
		0 & 0\leq t< 1-3\delta_2,\\
		\text{linear} & 1-3\delta_2\leq t< 1-2\delta_2,\\
		1 & 1-2\delta_2\leq t\leq 1.
	\end{cases}
	\]
    
    By Lemma \ref{approx}(1) and Lemma \ref{approx}(2), we can choose $\delta_3$ with $0<\delta_3<\delta_2$ such that if $y\in (A\rtimes_\alpha G)_+^1$, $z\in (A\rtimes_\alpha G)^1$ satisfy $\|yz-zy\|<\delta_3$, then $\|l_2(y)z-zl_2(y)\|<\delta_2$, and if $y,z\in (A\rtimes_\alpha G)_+^1$ satisfy $\|y-z\|<\delta_3$, then $\|l_2(y)-l_2(z)\|<\delta_2$.
    
    By choosing a bijection from $G$ to $\{1,2,\cdots,n\}$ which sends the identity of $G$ to 1, we identify $L(l^2(G))$ with $M_n$. Let $\{e_{g,h}\}_{g,h\in G}$ be a system of matrix units for $M_n$. Applying \cite[Lemma VII.3]{MR2712082} with $F_1$ in place of $F$, with $\delta_3/n$ in place of $\varepsilon$, and with $s_1$ in place of $x$, we get $\delta_4$ with $0<\delta_4<\delta_3$, orthogonal positive elements $f_g\in A_+^1$ for each $g\in G$, and a c.p.c. order zero map $\varphi\colon M_n\rightarrow A\rtimes_\alpha G$ such that
    \begin{enumerate}
    	\setcounter{enumi}{4}
    	\item\label{lem3-5} $\|\varphi(e_{g,g})x-x\varphi(e_{g,g})\|<\delta_3/n$ for all $x\in F_1$ and all $g\in G$.
    	\item\label{lem3-6} $\|u_g\varphi(e_{h,k})-\varphi(e_{gh,k})\|<\delta_3/n$ for all $g,h,k\in G$.
    	\item\label{lem3-7} $\varphi(e_{1,1})\in A_+^1\setminus\{0\}$, where 1 is the unit of $G$.
    	\item\label{lem3-8} $\|\varphi(e_{g,g})-f_g\|<\delta_3/n$ for all $g\in G$.
    	\item\label{lem3-9} $1-f\precsim_A s_1$, where $f=\sum_{g\in G}f_g$.
    	\item\label{lem3-10} $\|fs_1f\|>1-\delta_4$.
    \end{enumerate}
    Using the functional calculus for order zero maps, define c.p.c. order zero maps $\varphi_1,\varphi_2\colon M_n\rightarrow A\rtimes_\alpha G$ by $\varphi_1=l_1(\varphi)$ and $\varphi_2=l_2(\varphi)$. Using Theorem \ref{thmozs}, it is easily checked that
    \begin{equation}\label{eq31}
    	\varphi_1(e_{g_1,h_1})\varphi_2(e_{g_2,h_2})=\delta_{h_1,g_2}\varphi_2(e_{g_1,h_2})=\varphi_2(e_{g_1,h_1})\varphi_1(e_{g_2,h_2})
    \end{equation}
    for all $g_1,g_2,h_1,h_2\in G$.
    
    Using the fact that $\varphi$ is a c.p.c. order zero map at the first step, (\ref{lem3-8}) at the second step and (\ref{lem3-10}) at the fourth step, we get
    \begin{equation}\label{eq32}
    	\|\varphi(e_{1,1})\|=\|\varphi(1_n)\|>\|f\|-\delta_3\geq \|fs_1f\|-\delta_3>1-2\delta_2.
    \end{equation}
    Thus by Lemma \ref{lemcpf} and (\ref{lem3-7}), we have
    \[
    \varphi_2(e_{1,1})=l_2(\varphi(e_{1,1}))\in A_+^1\setminus\{0\}.
    \]
    Define a function $\psi\colon \overline{\varphi_2(e_{1,1})A\varphi_2(e_{1,1})}\otimes M_n\rightarrow A\rtimes_\alpha G$ by
    \[
    \psi(\sum_{g,h\in G}x_{g,h}\otimes e_{g,h})=\sum_{g,h\in G}\varphi_1(e_{g,1})x_{g,h}\varphi_1(e_{1,h})
    \]
    for all $x_{g,h}\in \overline{\varphi_2(e_{1,1})A\varphi_2(e_{1,1})}$ and all $g,h\in G$. Using the same argument as in \cite[Lemma VII.4]{MR2712082}, we conclude that  $\psi$ is a $\ast$-homomorphism.
    
    Set $B=\mathrm{Im}(\psi)\subset A\rtimes_\alpha G$, $d_1=\psi(\varphi_2(e_{1,1})\otimes 1_n)$ and $d=l_1(d_1)$. Since $A$ is simple, it follows that $\overline{\varphi_2(e_{1,1})A\varphi_2(e_{1,1})}\otimes M_n$ is simple. Therefore, $\psi$ is injective and $B\cong \overline{\varphi_2(e_{1,1})A\varphi_2(e_{1,1})}\otimes M_n$.
    Using (\ref{eq31}) at the second step, Lemma \ref{lemcpf} at the fourth step, we get
    \[
    d_1=\sum_{g\in G}\psi(\varphi_2(e_{1,1})\otimes e_{g,g})=\sum_{g\in G}\varphi_2(e_{g,g})=\varphi_2(1_n)=l_2(\varphi(1_n)).
    \]
    By (\ref{eq32}), we have $\|d_1\|=1$. Thus $\|d\|$=1. This proves (\ref{lem3-4}).
    
    Let $x\in F_1$ and $g\in G$. By (\ref{lem3-5}) and (\ref{lem3-6}), we have
    \[
    \|\varphi(1_n)x-x\varphi(1_n)\|\leq \sum_{h\in G}\|\varphi(e_{h,h})x-x\varphi(e_{h,h})\|<\delta_3,
    \]
    \[
    \|u_g\varphi(1_n)u_g^*-\varphi(1_n)\|\leq \sum_{h\in G}\|u_g\varphi(e_{h,h})u_g^*-\varphi(e_{gh,gh})\|<\delta_3.
    \]
    Thus by the choice of $\delta_3$, we get
    \begin{equation}\label{eq33}
    	\|\varphi_2(e_{g,g})x-x\varphi_2(e_{g,g})\|=\|l_2(\varphi(e_{g,g}))x-xl_2(\varphi(e_{g,g}))\|<\delta_2,
    \end{equation}
    \begin{equation*}
    	\|d_1x-xd_1\|=\|l_2(\varphi(1_n))x-xl_2(\varphi(1_n))\|<\delta_2,
    \end{equation*}
    \begin{align*}
    \|u_gd_1u^*_g-d_1\| & =\|u_gl_2(\varphi(1_n))u_g^*-l_2(\varphi(1_n))\|\\
    & =\|l_2(u_g\varphi(1_n)u_g^*)-l_2(\varphi(1_n))\|<\delta_2.
    \end{align*}
    Note that $\|d-d_1\|\leq 3\delta_2<\varepsilon/3$. Therefore, we have
    \begin{align*}
        \|dxu_g-xu_gd\| \leq & \|dxu_g-d_1xu_g\|+\|d_1xu_g-xd_1u_g\|\\
        & +\|xd_1u_g-xu_gd_1\|+\|xu_gd_1-xu_gd\|<2\varepsilon/3+2\delta_2<\varepsilon.
    \end{align*}
    This proves (\ref{lem3-1}).
    
    Note that
    \[
    \|\varphi_1(e_{g,h})-\varphi(e_{g,h})\|\leq \|l_1(\varphi(1_n))-\varphi(1_n)\|\leq 3\delta_2
    \]
    for all $g,h\in G$. Using (\ref{eq31}) at the first step and (\ref{lem3-6}) at the last step, we get
    \begin{align}\label{eq34}
    	\|u_g\varphi_2(e_{h_1,h_2})-\varphi_2(e_{gh_1,h_2})\|
    	= & \|u_g\varphi_1(e_{h_1,h_1})\varphi_2(e_{h_1,h_2})-\varphi_1(e_{gh_1,h_1})\varphi_2(e_{h_1,h_2})\| \notag\\
    	\leq & \|u_g\varphi_1(e_{h_1,h_1})-\varphi_1(e_{gh_1,h_1})\| \notag\\
        \leq &  \|u_g\varphi_1(e_{h_1,h_1})-u_g\varphi(e_{h_1,h_1})\|+\|u_g\varphi(e_{h_1,h_1})-\varphi(e_{gh_1,h_1})\| \notag\\
    	& +\|\varphi(e_{gh_1,h_1})-\varphi_1(e_{gh_1,h_1})\|< 6\delta_2+\delta_2/n
    \end{align}
    for all $g,h_1,h_2\in G$.
    
    Let $x\in F_1$. Using (\ref{eq33}) at the third step, (\ref{eq34}) at the fourth step, we have
    \begin{align*}
    	& \|d_1xd_1-\sum_{g\in G}\psi((\varphi_2(e_{1,1})\alpha_{g^{-1}}(x)\varphi_2(e_{1,1}))\otimes e_{g,g})\|\\
    	\leq & \|d_1xd_1-\sum_{g\in G}\varphi_2(e_{g,g})x\varphi_2(e_{g,g})\|+\sum_{g\in G}\|\varphi_2(e_{g,g})x\varphi_2(e_{g,g})-\psi((\varphi_2(e_{1,1})\alpha_{g^{-1}}(x)\varphi_2(e_{1,1}))\otimes e_{g,g})\|\\
    	\leq & \sum_{g,h\in G, g\neq h}\|\varphi_2(e_{g,g})x\varphi_2(e_{h,h})\|+\sum_{g\in G}\|\varphi_2(e_{g,g})x\varphi_2(e_{g,g})-\varphi_2(e_{g,1})\alpha_{g^{-1}}(x)\varphi_2(e_{1,g})\|\\
    	\leq & n(n-1)\delta_2 +\sum_{g\in G}\|\varphi_2(e_{g,g})x\varphi_2(e_{g,g})-\varphi_2(e_{g,1})u_g^*xu_g\varphi_2(e_{1,g})\|\\
    	< & (n^2+11n+2)\delta_2.
    \end{align*}
    Let $g\in G$. Using (\ref{eq31}) at the first step, (\ref{eq34}) at the second step, we have
    \begin{align*}
    \|u_gd_1-\sum_{h\in G}\psi(\varphi_2(e_{1,1})\otimes e_{gh,h})\| & \leq \sum_{h\in G}\|u_g\varphi_2(e_{h,h})-\varphi_2(e_{gh,h})\|<(6n+1)\delta_2.
    \end{align*}
    Thus
    \[
    d_1^3xu_g\approx_{2\delta_2}(d_1xd_1)(u_gd_1)\in_{(n^2+17n+3)\delta_2} B.
    \]
    Since $2\delta_2+(n^2+17n+3)\delta_2<\delta_1$, by the choice of $\delta_1$ we get
    \[
    dxu_g\approx_{\varepsilon/3}d_1xu_g\in_{\varepsilon/3} B.
    \]
    This proves (\ref{lem3-2}).
    
    Define a continuous function $l_3\colon [0,1]\rightarrow [0,1]$ by $l_3(t)=1-l_2(1-t)$. By (\ref{lem3-8}) and the choice of $\delta_3$, we have
    \[
    \|1-d_1-(1-l_2(f))\|=\|l_2(\varphi(1_n))-l_2(f)\|<\delta_2.
    \]
    Using Lemma \ref{lemcuc}(2) at the fourth step, Lemma \ref{lemcuc}(1) at the sixth step, and (\ref{lem3-9}) at the seventh step, we get
    \begin{align*}
    1-d & =1-l_1(d_1)=\frac{1}{1-3\delta_2}(1-d_1-3\delta_2)_+\sim (1-d_1-3\delta_2)_+\\
    & \precsim 1-l_2(f)=l_3(1-f)\precsim 1-f \precsim s_1\precsim s.
    \end{align*}
    This proves (\ref{lem3-3}).
\end{proof}

Condition (\ref{lem3-4}) of Lemma \ref{lemap} can be replaced by a seemingly stronger equivalent condition.

\begin{theorem}\label{thmap}
	Let $A$ be an infinite-dimensional simple unital C*-algebra, let $G$ be a finite group with card$(G)=n$, and let $\alpha\colon G\rightarrow \mathrm{Aut}(A)$ be an action of $G$ on $A$ which has the weak tracial Rokhlin property. Then for any finite subet $F\subset A\rtimes_\alpha G$, any $\varepsilon>0$, any $s\in (A\rtimes_\alpha G)_+\setminus \{0\}$, there exist an element $a\in A_+^1$, a C*-subalgebra $B$ of $A\rtimes_\alpha G$ with $B\cong \overline{aAa}\otimes M_n$, and an element $d\in B_+^1$, such that
	\begin{enumerate}
		\item $\|dx-xd\|<\varepsilon$ for all $x\in F$.
		\item $dx\in_\varepsilon B$ for all $x\in F$.
		\item $1-d\precsim_{A\rtimes_\alpha G} s$.
		\item $\|dxd\|>\|x\|-\varepsilon$ for all $x\in F$.
	\end{enumerate}
\end{theorem}

\begin{proof}
	Let $F\subset A\rtimes_\alpha G$ be a finite subset, let $\varepsilon>0$, and let $s\in (A\rtimes_\alpha G)_+\setminus \{0\}$. Without loss of generality, we may assume that $1\in F$. Since $A\rtimes_\alpha G$ is simple, by \cite[Proposition 2.2]{MR0511256} there is a separable simple unital C*-subalgebra $C$ of $A\rtimes_\alpha G$ which contains $F$.
	
	Fix a dense sequence $(c_k)_{k=1}^\infty$ in $C$. For each $k\in \mathbb{N}$, applying Lemma \ref{lemap} with $S_k=\{c_1,c_2,\cdots,c_k\}$ in place of $F$, $1/k$ in place of $\varepsilon$, and $s$ as given, we get a sequence of C*-subalgebras $B_k \subset A\rtimes_\alpha G$ with $B_k\cong \overline{a_kAa_k}\otimes M_n$ for some $a_k\in A_+$, and a sequence of positive elements $d_k\in B_k$ such that
	\begin{enumerate}
		\setcounter{enumi}{4}
		\item\label{thm35} $\|d_kx-xd_k\|<1/k$ for all $x\in S_k$.
		\item\label{thm36} $d_kx\in_{1/k} B_k$ for all $x\in S_k$.
		\item\label{thm37} $1-d_k\precsim_{A\rtimes_\alpha G} s$.
		\item\label{thm38} $\|d_k\|=1$.
	\end{enumerate}
	
	Let $\pi\colon l^\infty(A\rtimes_\alpha G)\rightarrow l^\infty(A\rtimes_\alpha G)/c_0(A\rtimes_\alpha G)$ be the quotient map. Define a c.p.c. map $\varphi\colon C\rightarrow l^\infty(A\rtimes_\alpha G)/c_0(A\rtimes_\alpha G)$ by
	\[
	\varphi(x)=\pi(\{d_1xd_1, d_2xd_2, \cdots\})
	\]
	for all $x\in C$. By (\ref{thm35}), (\ref{thm36}), and the density of $(c_k)_{k=1}^\infty$ in $C$, we have
	\begin{enumerate}
		\setcounter{enumi}{8}
		\item\label{thm39} $\lim\limits_{k\rightarrow \infty}\|d_kx-xd_k\|=0$ for all $x\in C$.
		\item\label{thm310} $\lim\limits_{k\rightarrow \infty}\mathrm{dist}(d_kx, B_k)=0$ for all $x\in C$.
	\end{enumerate}
	By (\ref{thm38}), we have
	\[
	\|\varphi\|=\|\varphi(1)\|=\limsup_{k\rightarrow \infty}\|d_k^2\|=1.
	\]
	Let $x,y\in C_+^1$ satisfy $xy=0$. Then we have
	\begin{align*}
	\|\varphi(x)\varphi(y)\| & =\limsup_{k\rightarrow \infty}\|d_kxd_k^2yd_k\| \\
	& \leq \limsup_{k\rightarrow \infty}\|xd_k^2-d_k^2x\|+\limsup_{k\rightarrow \infty}\|d_k^3xyd_k\|=0.
	\end{align*}
	Therefore, $\varphi$ is a c.p.c. order zero map. Since $C$ is simple, it follows from \cite[Proposition 5.3]{MR4464578} that
	\[
	\|\varphi(x)\|=\|\varphi\|\|x\|=\|x\|
	\]
	for all $x\in C$. Hence there exists a subsequence $(k_l)_{l=1}^\infty$ such that
	\begin{enumerate}
	\setcounter{enumi}{10}
	\item\label{thm311} $\|d_{k_l}xd_{k_l}\|>\|x\|-\varepsilon$ for all $x\in F$ and $l\geq 1$.
    \end{enumerate}
    
    Set $B=B_{k_l}$ and $d=d_{k_l}$ for some sufficiently large $l$. Then the conclusion of the theorem follows from (\ref{thm39}), (\ref{thm310}), (\ref{thm37}), and (\ref{thm311}).
\end{proof}

The following lemma will be used in the proof of the main theorem. It says that the approximation structure in Theorem \ref{thmap} is preserved under tensoring with matrix algebras.
\begin{lemma}\label{lemapm}
	Let $A$ be an infinite-dimensional simple unital C*-algebra, let $G$ be a finite group with card$(G)=n$, and let $\alpha\colon G\rightarrow \mathrm{Aut}(A)$ be an action of $G$ on $A$ which has the weak tracial Rokhlin property. Let $m\in \mathbb{N}$. Then for any finite subet $F\subset M_m(A\rtimes_\alpha G)$, any $\varepsilon>0$, any $s\in (M_m(A\rtimes_\alpha G))_+\setminus \{0\}$, there exist an element $a\in A_+^1$, a C*-subalgebra $B$ of $M_m(A\rtimes_\alpha G)$ with $B\cong \overline{aAa}\otimes M_{nm}$, and an element $d\in B_+^1$, such that
	\begin{enumerate}
		\item\label{lem3.51} $\|dx-xd\|<\varepsilon$ for all $x\in F$.
		\item\label{lem3.52} $dx\in_\varepsilon B$ for all $x\in F$.
		\item\label{lem3.53} $1_{M_m(A\rtimes_\alpha G)}-d\precsim_{M_m(A\rtimes_\alpha G)} s$.
		\item\label{lem3.54} $\|dxd\|>\|x\|-\varepsilon$ for all $x\in F$.
	\end{enumerate}
\end{lemma}

\begin{proof}
	Let $F=\{f^{(1)}, f^{(2)}, \cdots, f^{(l)}\}$ be a finite subset of $M_m(A\rtimes_\alpha G)$, let $\varepsilon>0$, and let $s\in (M_m(A\rtimes_\alpha G))_+\setminus \{0\}$. Choose $t\in (A\rtimes_\alpha G)_+\setminus\{0\}$ such that $t\precsim_{M_m(A\rtimes_\alpha G)} s$. Since $A\rtimes_\alpha G$ is simple unital and infinite-dimensional, by \cite[Lemma 2.4]{1408.5546} there exist pairwise orthogonal positive elements $t_1, t_2, \cdots, t_m\in \overline{t(A\rtimes_\alpha G)t}$ such that
	\[
	t_1\sim t_2\sim\cdots \sim t_m.
	\]
	
	Write $f^{(k)}=\sum_{i,j=1}^mf^{(k)}_{i,j}\otimes e_{i,j}$, where $\{e_{i,j}\}_{i,j=1}^m$ is a system of matrix units for $M_m$, $f^{(k)}_{i,j}\in A\rtimes_\alpha G$, $1\leq k\leq l$, $1\leq i,j\leq m$. Set $F_1=\{f^{(k)}_{i,j}\colon 1\leq k\leq l, 1\leq i,j\leq m\}$. Applying Lemma \ref{lemap} with $F_1$ in place of $F$, with $\varepsilon/m^2$ in place of $\varepsilon$, and with $t_1$ in place of $s$, we get a C*-subalgebra $B_1$ of $A\rtimes_\alpha G$ with $B_1\cong \overline{aAa}\otimes M_n$ for some $a\in A_+$, and an element $d_1\in (B_1)_+$ such that
	\begin{enumerate}
		\setcounter{enumi}{4}
		\item\label{lem3.55} $\|d_1x-xd_1\|<\varepsilon/m^2$ for all $x\in F_1$.
		\item\label{lem3.56} $d_1x\in_{\varepsilon/m^2} B_1$ for all $x\in F_1$.
		\item\label{lem3.57} $1_{A\rtimes_\alpha G}-d_1\precsim_{A\rtimes_\alpha G} t_1$.
		\item\label{lem3.58} $\|d_1\|=1$.
	\end{enumerate}
	
	Set $B=B_1\otimes M_m\cong \overline{aAa}\otimes M_{nm}$ and $d=d_1\otimes 1_m\in B^1_+$. Then (\ref{lem3.51}) follows from (\ref{lem3.55}), and (\ref{lem3.52}) follows from (\ref{lem3.56}).
	Using (\ref{lem3.57}) and Lemme \ref{lemcuc}(5) at the second step, Lemma \ref{lemcuc}(4) at the third step, and Lemma \ref{lemcuc}(3) at the fourth step, we get
	\[
	\begin{aligned}
	    1_{M_m(A\rtimes_\alpha G)}-d & =(1_{A\rtimes_\alpha G}-d_1)\otimes 1_m\\
	    & \precsim t_1\oplus t_2\oplus \cdots\oplus t_m \sim t_1+t_2+\cdots+t_m \precsim t\precsim s.
	\end{aligned}
	\]
	This proves (\ref{lem3.53}). Since $\|d_1\|=1$, we have $\|d\|=1$. Condition (\ref{lem3.54}) follows from the same argument used in the proof of Theorem \ref{thmap}.
\end{proof}

\section{Preservation of property (TM) and stable rank one}\label{sec4}

In this section, we give the proof of the main theorem. Before that, we give some basic lemmas.

\begin{lemma}\label{lemsfqtc}
	Let $A$ be a stably finite unital C*-algebra, let $G$ be a finite group, and let $\alpha$ be an action of $G$ on $A$. Then $A\rtimes_\alpha G$ is stably finite and $\mathrm{QT}(A\rtimes_\alpha G)\neq\emptyset$.
\end{lemma}
\begin{proof}
	Put $\mathrm{card}(G)=n$. Identifying $L(l^2(G))$ with $M_n$ and let $\{e_{g,h}\}_{g,h\in G}$ be a system of matrix units for $M_n$. By \cite[Corollary 4.1.6]{MR2391387}, the map $\varphi\colon A\rtimes_\alpha G\rightarrow A\otimes M_n$ defined by
	\[
	\varphi(\sum_{g\in G}a_gu_g)=\sum_{g,h\in G}\alpha_{h^{-1}}(a_g)\otimes e_{h,g^{-1}h}
	\]
	is a unital injective $\ast$-homomorphism. Thus $A\rtimes_\alpha G$ can be viewed as a unital C*-subalgebra of $M_n(A)$. Since $A$ is stably finite, we get $A\rtimes_\alpha G$ is stably finite. It follows from Theorem \ref{thmsfqt} that $\mathrm{QT}(A\rtimes_\alpha G)\neq\emptyset$.
\end{proof}

\begin{lemma}\label{lemiso}
	Let $A$ be a C*-algebra and let $x\in A$. Put $a=x^*x$ and $b=xx^*$. Then there is an isomorphism $\varphi\colon \overline{aAa}\rightarrow \overline{bAb}$ such that $\varphi(a)=b$ and $\tau(\varphi(z))=\tau(z)$ for any quasitrace $\tau$ on $A$ and any $z\in (\overline{aAa})_+$.
\end{lemma}
\begin{proof}
	Let $x=v|x|$ be the polar decomposition of $x$ in $A^{**}$. Then $vy\in A$ for every $y\in \overline{aAa}$ (see \cite[Lemma 3.5.1]{MR1884366}). Therefore $\varphi(y)=vyv^*$ is an isomorphism from $\overline{aAa}$ to $\overline{bAb}$ which maps $a$ to $b$. Let $\tau$ be a quasitrace on $A$ and let $z\in (\overline{aAa})_+$. Since $v^*v$ is the range projection of $|x|$ in $A^{**}$, we have
	\[
	\tau(\varphi(z))=\tau(vzv^*)=\tau(z^{1/2}v^*vz^{1/2})=\tau(z).
	\]
\end{proof}

\begin{lemma}\label{lemqt}
	Let $A$ be a simple unital C*-algebra with $\mathrm{QT}(A)\neq\emptyset$ and let $a\in \mathrm{Ped}(A\otimes K)_+$. Then $\sup\{\|\tau\vert_{\overline{a(A\otimes K)a}}\| \colon\tau \in \mathrm{QT}(A)\}<\infty$.
\end{lemma}
\begin{proof}
	Since $A$ is simple, $A\otimes K$ is simple. Then by \cite[Proposition II.5.4.3(ii)]{MR2188261}, there exist $n\in \mathbb{N}$ and $x_1,x_2,\cdots,x_n\in A\otimes K$ such that
	\[
	a=\sum_{i=1}^nx_i^*1_Ax_i.
	\]
	
	For any $\tau\in \mathrm{QT}(A)$, using Lemma \ref{lemqtnd} at the first step, we have
	\begin{align*}
		\|\tau\vert_{\overline{a(A\otimes K)a}}\|=d_\tau(a)\leq \sum_{i=1}^nd_\tau(x_i^*1_Ax_i)=\sum_{i=1}^nd_\tau(1_Ax_ix_i^*1_A)\leq nd_\tau(1_A)=n.
	\end{align*}
\end{proof}

Recall that for simple C*-algebra $A$, we identify every $\tau\in \widetilde{\mathrm{QT}}(\overline{aAa})$ with its unique extension $\tilde{\tau} \in \widetilde{\mathrm{QT}}(A)$.

\begin{lemma}\label{lemqte}
	Let $A$ be a simple unital C*-algebra with $\mathrm{QT}(A)\neq\emptyset$, let $m\in \mathbb{N}$, and let $e\in M_m(A)_+\setminus\{0\}$. Then $\sup\{\|\tau\vert_A\|\colon \tau\in \mathrm{QT}(\overline{eM_m(A)e})\}<\infty$.
\end{lemma}
\begin{proof}
	Since $M_m(A)$ is simple, there exist $n\in \mathbb{N}$ and $x_1, x_2, \cdots, x_n\in M_m(A)$ such that
	\[
	1_A=\sum_{i=1}^nx_i^*ex_i.
	\]
	
	For any $\tau\in \mathrm{QT}(\overline{eM_m(A)e})$, we have
	\[
	\|\tau\vert_A\|=d_\tau(1_A)\leq nd_\tau(e)=n\|\tau\vert_{\overline{eM_m(A)e}}\|=n.
	\]
\end{proof}

\begin{theorem}\label{thmtm}
	Let $A$ be an infinite-dimensional stably finite simple unital C*-algebra, let $G$ be a finite group, and let $\alpha\colon G\rightarrow \mathrm{Aut}(A)$ be an action of $G$ on $A$ which has the weak tracial Rokhlin property. If $A$ has property (TM), then $A\rtimes_\alpha G$ has property (TM).
\end{theorem}
\begin{proof}
	Put $B=A\rtimes_\alpha G$. Then $B$ is simple by Remark \ref{remsim} and $\mathrm{QT}(B)\neq\emptyset$ by Lemma \ref{lemsfqtc}. Let $b\in \mathrm{Ped}(B\otimes K)_+^1$, let $k\in \mathbb{N}$, and let $\varepsilon>0$. We have to show that there exists a c.p.c. order zero $\varphi\colon M_k\rightarrow \overline{b(B\otimes K)b}$ such that
	\[
	\|b-\varphi(1_k)b\|_{2,\mathrm{QT}(B)}\leq\varepsilon.
	\]
	
	Set $L_1=\sup \{\|\tau\vert_{\overline{b(B\otimes K)b}}\|\colon \tau\in \mathrm{QT}(B)\}$. Then $L_1<\infty$ by Lemma \ref{lemqt}. Let $\{e_{i,j}\}$ be the standard matrix units for $K$. For each $n\in \mathbb{N}$, set $p_n=\sum_{i=1}^n1_A\otimes e_{i,i}$. Then $(p_n)_{n=1}^\infty$ is an increasing approximate identity of projections for $A\otimes K$. Choose a large $m\in \mathbb{N}$ such that
	\begin{equation}\label{eq1}
		\|b-b^{1/2}p_mb^{1/2}\|<\frac{\varepsilon}{8\sqrt{L_1}}.
	\end{equation}
	
	Set $b_1=b^{1/2}p_mb^{1/2}$ and $b_2=p_mbp_m$. Then $b_1\in (\overline{b(B\otimes K)b})_+^1$ and $b_2\in M_m(B)_+^1$. Since $b_2\in M_m(B)$, we have $\overline{b_2(B\otimes K)b_2}\cong \overline{b_2M_m(B)b_2}$. Applying Lemma \ref{lemiso} with $x=b^{1/2}p_m$, we obtain an isomorphism $\varphi_1\colon \overline{b_2M_m(B)b_2}\rightarrow \overline{b_1(B\otimes K)b_1}$ such that $\varphi_1(b_2)=b_1$ and
	\begin{equation}\label{eq2}
		\tau_1(\varphi_1(z))=\tau_1(z) \text{ for all } \tau_1 \in \widetilde{\mathrm{QT}}(B) \text{ and all } z\in \overline{b_2M_m(B)b_2}.
	\end{equation}
	
	Choose $\varepsilon_1$ such that
	\[
	0<\varepsilon_1<\frac{\varepsilon}{24\sqrt{m}}.
	\]
	Applying Lemma \ref{lemcup} with $\varepsilon_1$ in place of $\varepsilon$ and with $\sigma=\varepsilon_1$, we get $\delta_1$ with $0<\delta_1<\varepsilon_1$ such that if $x,y\in M_m(B)^1_+$ satisfying $\|x-y\|<\delta_1$, then there exists an injective $\ast$-homomorphism $\psi\colon \overline{(x-\varepsilon_1)_+M_m(B)(x-\varepsilon_1)_+}\rightarrow \overline{yM_m(B)y}$ satisfying
	\begin{equation}\label{eq3}
		\|\psi(z)-z\|<\varepsilon_1\|z\| \text{ for all } z\in \overline{(x-\varepsilon_1)_+M_m(B)(x-\varepsilon_1)_+}.
	\end{equation}
	By Lemma \ref{approx}(2), we can choose $\delta_2$ such that if $x,y$ are self-adjoint elements of $M_m(B)^1$ satisfying $\|x-y\|<\delta_2$, then $\|x_+-y_+\|<\delta_1$. Set $\delta=\min\{\delta_1,\delta_2/4\}$.
	
	Note that $M_m(B)$ is simple unital and infinite-dimensional. By \cite[Corollary 2.5]{1408.5546} there is an element $s\in M_m(B)_+\setminus\{0\}$ such that for all $\tau_2\in \mathrm{QT}(M_m(B))$ we have
	\begin{equation}\label{eq4}
		d_{\tau_2}(s)<\frac{\varepsilon^2}{64m}.
	\end{equation}
	Set $n=\mathrm{card}(G)$. Applying Lemma \ref{lemapm} with $F=\{b_2, b_2^{1/2}\}$, with $\delta$ in place of $\varepsilon$, and with $s, m$ be given above, we obtain an element $a\in A_+^1$, a C*-subalgebra $C$ of $M_m(B)$, an isomorphism $\varphi_3\colon M_{nm}(\overline{aAa})\rightarrow C$, and an element $d\in C_+^1$ such that
	\begin{flalign}\label{eq5}
		& (1)\ \|db_2^{1/2}-b_2^{1/2}d\|<\delta. &\notag\\
		& (2)\ db_2\in_\delta C. &\notag\\
		& (3)\ 1_{M_m(B)}-d\precsim_{M_m(B)} s.
	\end{flalign}
	
	By (1) and (2), there exists an element $c\in C$ such that $\|b_2^{1/2}db_2^{1/2}-c\|<2\delta$. Set
	\[
	c_1=\frac{c+c^*}{2(1+2\delta)},\text{ }c_2=(c_1)_+. 
	\]
	Then $c_1$ is a self-adjoint element of $C^1$ and
	\[
	\|c_1-b_2^{1/2}db_2^{1/2}\|\leq\|c_1-\frac{c+c^*}{2}\|+\|\frac{c+c^*}{2}-b_2^{1/2}db_2^{1/2}\|<4\delta<\delta_2.
	\]
	By the choice of $\delta_2$ and the fact that $b_2^{1/2}db_2^{1/2}$ is a positive element, we have
	\begin{equation}\label{eq6}
		\|c_2-b_2^{1/2}db_2^{1/2}\|<\delta_1.
	\end{equation}
	Then by the choice of $\delta_1$, we get an injective $\ast$-homomorphism $\varphi_2\colon \overline{(c_2-\varepsilon_1)_+C(c_2-\varepsilon_1)_+}\rightarrow \overline{(b_2^{1/2}db_2^{1/2})M_m(B)(b_2^{1/2}db_2^{1/2})}\subset \overline{b_2M_m(B)b_2}$ satisfying
	\begin{equation}\label{eq7}
		\|\varphi_2(z)-z\|<\varepsilon_1\|z\| \text{ for all } z\in \overline{(c_2-\varepsilon_1)_+C(c_2-\varepsilon_1)_+}.
	\end{equation}
	
	Set $e=\varphi_3^{-1}((c_2-\varepsilon_1)_+)\in M_{nm}(\overline{aAa})$. If $e\neq 0$, set $L_2=\sup\{\|\tau\vert_A\|\colon \tau\in \mathrm{QT}(\overline{eM_{nm}(A)e})\}$. Then $L_2<\infty$ by Lemma \ref{lemqte}. Since $A$ has property (TM), applying Definition \ref{deftm} with $e$ in place of $a$, with $k$ in place of $n$, and with $\frac{\varepsilon}{8\sqrt{mL_2}}$ in place of $\varepsilon$, we obtain a c.p.c. order zero map $\varphi_4\colon M_k\rightarrow \overline{e(A\otimes K)e}\cong\overline{eM_{nm}(A)e}=\overline{eM_{nm}(\overline{aAa})e}$ such that
	\begin{equation}\label{eq8}
		\|e-\varphi_4(1_k)e\|_{2, \mathrm{QT}(A)}<\frac{\varepsilon}{8\sqrt{mL_2}}.
	\end{equation}
	If $e=0$, we set $\varphi_4=0$.

	Let $\varphi$ be the composition
	\begin{align*}
		M_k & \xrightarrow{\varphi_4} \overline{eM_{nm}(\overline{aAa})e}\xrightarrow{\varphi_3} \overline{(c_2-\varepsilon_1)_+C(c_2-\varepsilon_1)_+}\\
		& \xrightarrow{\varphi_2}\overline{b_2M_m(B)b_2} \xrightarrow{\varphi_1}\overline{b_1(B\otimes K)b_1} \subset \overline{b(B\otimes K)b}.
	\end{align*}
    Then $\varphi$ is a c.p.c. order zero map from $M_k$ to $\overline{b(B\otimes K)b}$. We estimate $\|b-\varphi(1_k)b\|_{2,\mathrm{QT}(B)}$.
    
	Let $\tau\in \mathrm{QT}(B)$. By Lemma \ref{lemqtn}(2) we have
	\begin{equation}\label{eq9}
		\|b-\varphi(1_k)b\|_{2,\tau}^{2/3}\leq \|b_1-\varphi(1_k)b_1\|_{2, \tau}^{2/3}+\|(b-b_1)-\varphi(1_k)(b-b_1)\|_{2, \tau}^{2/3}.
	\end{equation}
	Note that $b-b_1$ is a positive element of $\overline{b(B\otimes K)b}$. Using the definition of $L_1$ at the third step, and (\ref{eq1}) at the fourth step, we have
	\begin{align}\label{eq10}
	\|(b-b_1)-\varphi(1_k)(b-b_1)\|_{2, \tau}^{2/3} & =[\tau((b-b_1)(1_{\mathcal{M}(B\otimes K)}-\varphi(1_k))^2(b-b_1))]^{1/3}\notag\\
	& \leq [\tau((b-b_1)^2)]^{1/3}\leq (L_1\|b-b_1\|^2)^{1/3}<\frac{1}{4}\varepsilon^{2/3}.
    \end{align}
	Note that $b_1=\varphi_1(b_2)$. Using (\ref{eq2}) at the second step, we have
	\begin{align}\label{eq11}
	\|b_1-\varphi(1_k)b_1\|_{2, \tau}^2 &= \tau\circ\varphi_1((b_2-\varphi_2\circ\varphi_3\circ\varphi_4(1_k)b_2)^*(b_2-\varphi_2\circ\varphi_3\circ\varphi_4(1_k)b_2))\notag\\
	& =\tau((b_2-\varphi_2\circ\varphi_3\circ\varphi_4(1_k)b_2)^*(b_2-\varphi_2\circ\varphi_3\circ\varphi_4(1_k)b_2))\notag\\
	& =\|b_2-\varphi_2\circ\varphi_3\circ\varphi_4(1_k)b_2\|_{2,\tau}^2.
	\end{align}
	
	Note that $b_2=b_2^{1/2}db_2^{1/2}+b_2^{1/2}(1_{M_m(B)}-d)b_2^{1/2}$. Using Lemma \ref{lemqtn}(2) again, we have
	\begin{align}\label{eq12}
		\|b_2-\varphi_2\circ\varphi_3\circ\varphi_4(1_k)b_2\|_{2,\tau}^{2/3}\leq &  \|(1_{M_m(B)}-\varphi_2\circ\varphi_3\circ\varphi_4(1_k))b_2^{1/2}db_2^{1/2}\|_{2,\tau}^{2/3}\notag\\
		&+ \|(1_{M_m(B)}-\varphi_2\circ\varphi_3\circ\varphi_4(1_k))b_2^{1/2}(1_{M_m(B)}-d)b_2^{1/2}\|_{2,\tau}^{2/3}.
	\end{align}
	Note that $\tau/m\in \mathrm{QT}(M_m(B))$. Using Lemma \ref{lemqtn}(3) at the first step, (\ref{eq5}) at the fifth step, and (\ref{eq4}) at the last step, we get
	\begin{align}\label{eq13}
	& \|(1_{M_m(B)}-\varphi_2\circ\varphi_3\circ\varphi_4(1_k))b_2^{1/2}(1_{M_m(B)}-d)b_2^{1/2}\|_{2,\tau}^{2/3}\notag\\
	\leq &  \|(1_{M_m(B)}-\varphi_2\circ\varphi_3\circ\varphi_4(1_k))b_2^{1/2}\|^{2/3}\|1_{M_m(B)}-d\|_{2,\tau}^{2/3}\|b_2^{1/2}\|^{2/3}\notag\\
	\leq & \|1_{M_m(B)}-d\|_{2,\tau}^{2/3}=[\tau((1_{M_m(B)}-d)^2)]^{1/3}\notag\\
	\leq & [d_\tau(1_{M_m(B)}-d)]^{1/3}\leq (d_\tau(s))^{1/3}<\frac{1}{4}\varepsilon^{2/3}.
	\end{align}
	By (\ref{eq6}) and (\ref{eq7}), we have
	\begin{align*}
	\|b_2^{1/2}db_2^{1/2}-\varphi_2((c_2-\varepsilon_1)_+)\|\leq & \|b_2^{1/2}db_2^{1/2}-c_2\|+\|c_2-(c_2-\varepsilon_1)_+\|\\
	& +\|(c_2-\varepsilon_1)_+-\varphi_2((c_2-\varepsilon_1)_+)\|< 3\varepsilon_1.
    \end{align*}
	Thus using Lemma \ref{lemqtn}(2) and Lemma \ref{lemqtn}(3) again, we get
    \begin{align}\label{eq14}
	& \|(1_{M_m(B)}-\varphi_2\circ\varphi_3\circ\varphi_4(1_k))b_2^{1/2}db_2^{1/2}\|_{2,\tau}^{2/3}\notag\\
	\leq &  \|(1_{M_m(B)}-\varphi_2\circ\varphi_3\circ\varphi_4(1_k))\varphi_2((c_2-\varepsilon_1)_+)\|_{2,\tau}^{2/3}+(3\varepsilon_1)^{2/3}\|1_{M_m(B)}-\varphi_2\circ\varphi_3\circ\varphi_4(1_k)\|_{2,\tau}^{2/3}\notag\\
	\leq & \|(1_{M_m(B)}-\varphi_2\circ\varphi_3\circ\varphi_4(1_k))\varphi_2((c_2-\varepsilon_1)_+)\|_{2,\tau}^{2/3}+(3\varepsilon_1)^{2/3}m^{1/3}\notag\\
	< & |(1_{M_m(B)}-\varphi_2\circ\varphi_3\circ\varphi_4(1_k))\varphi_2((c_2-\varepsilon_1)_+)\|_{2,\tau}^{2/3}+ \frac{1}{4}\varepsilon^{2/3}.
    \end{align}
    
    Set $\mu=\tau\circ\varphi_2\circ\varphi_3\vert_{\overline{eM_{nm}(A)e}}$. Then $\mu$ is a bounded 2-quasitrace on $\overline{eM_{nm}(A)e}$. It is easily checked that $\|\mu\vert_{\overline{eM_{nm}(A)e}}\|\leq m$. If $e\neq 0$, we view $\mu$ as a bounded 2-quasitrace on $M_{nm}(A)$ by Remark \ref{remqt}(3). Note that $(c_2-\varepsilon_1)_+=\varphi_3(e)$. Then by the definition of $L_2$ and (\ref{eq8}), we get 
    \begin{align}\label{eq15}
    \|(1_{M_m(B)}-\varphi_2\circ\varphi_3\circ\varphi_4(1_k))\varphi_2((c_2-\varepsilon_1)_+)\|_{2,\tau}^{2/3} = \|e-\varphi_4(1_k)e\|_{2,\mu}^{2/3}<\frac{1}{4}\varepsilon^{2/3}.
    \end{align}
    If $e=0$, the above inequality clearly holds.
    
    Combining (\ref{eq9}), (\ref{eq10}), (\ref{eq11}), (\ref{eq12}), (\ref{eq13}), (\ref{eq14}), and (\ref{eq15}), we get
    \begin{equation*}
    	\|b-\varphi(1_k)b\|_{2,\tau}<\varepsilon.
    \end{equation*}
    Since $\tau\in \mathrm{QT}(B)$ is arbitrary, we conclude that
    \begin{equation*}
    	\|b-\varphi(1_k)b\|_{2,\mathrm{QT}(B)}\leq\varepsilon.
    \end{equation*}
\end{proof}

\begin{definition}
	Let $A$ be a simple unital C*-algebra. Recall that $A$ is said to have strict comparison if whenever $a,b\in (A\otimes K)_+$ satisfy $d_\tau(a)<d_\tau(b)$ for all $\tau\in \mathrm{QT}(A)$, then $a\precsim_A b$.
\end{definition}

\begin{remark}
	In some references, the definition of strict comparison is defined for positive elements in $\cup_{n=1}^\infty M_n(A)$ as follows:
	
	A simple unital C*-algebra $A$ is said to have strict comparison if whenever $a,b\in \cup_{n=1}^\infty M_n(A)_+$ satisfy $d_\tau(a)<d_\tau(b)$ for all $\tau\in \mathrm{QT}(A)$, then $a\precsim_A b$.
	
	By \cite[Proposition 6.12]{1408.5546}, this definition is consistent with the previous one.
\end{remark}

Let $\alpha\colon G\rightarrow \mathrm{Aut}(A)$ be an action of a finite group $G$ on a unital C*-algebra $A$. Denote by $A^\alpha$ the fixed point algebra defined by
\[
A^\alpha=\{a\in A\colon \alpha_g(a)=a \text{ for all } g\in G\}.
\]

\begin{corollary}\label{corsr}
	Let $A$ be an infinite-dimensional separable simple unital C*-algebra, let $G$ be a finite group, and let $\alpha\colon G\rightarrow \mathrm{Aut}(A)$ be an action of $G$ on $A$ which has the weak tracial Rokhlin property. If $A$ has stable rank one and strict comparison, then $A\rtimes_\alpha G$ and $A^\alpha$ have stable rank one.
\end{corollary}
\begin{proof}
	Since $A$ has stable rank one, $A$ is stably finite. Then $\mathrm{QT}(A)\neq\emptyset$ by Theorem \ref{thmsfqt}. Since $A$ has stable rank one and strict comparison, it follows from \cite[Theorem 1.1]{fu2024tracialoscillationzerostable} that $A$ has property (TM). Now consider the algebra $A\rtimes_\alpha G$. It is separable since $A$ is separable and $G$ is a finite group. It is simple by Remark \ref{remsim}. It has a non-trivial bounded 2-quasitrace by Lemma \ref{lemsfqtc}. It has strict comparison by \cite[Theorem 4.5]{MR4336489} and property (TM) by Theorem \ref{thmtm}. Therefore, it follows from \cite[Theorem 1.1]{fu2024tracialoscillationzerostable} that $A\rtimes_\alpha G$ has stable rank one. By \cite[Lemma 4.3(5)]{MR4336489}, $A^\alpha$ also has stable rank one.
\end{proof}


\section*{Acknowledgements}
The authors would like to thank Xuanlong Fu for many helpful discussions.




\end{document}